\documentclass[12pt,a4paper]{amsart}

\usepackage{amsmath}
\usepackage{amssymb}

\theoremstyle{plain} 

\newtheorem{theorem}{Theorem}[section]   
\newtheorem{fact}[theorem]{Fact}

\newtheorem{lemma}[theorem]{Lemma}         
\newtheorem{proposition}[theorem]{Proposition}  

\theoremstyle{definition}

\theoremstyle{remark}

\numberwithin{equation}{section}

\newcommand{\N}{{\mathbb N}}

\newcommand{\osc}{\operatorname{osc}}
\newcommand{\ep}{\varepsilon}

\newcommand{\vf}{\varphi}

\newcommand{\diam}{\operatorname{diam}}

\newcommand{\B}{{\mathcal{B}}}
\newcommand{\F}{{\mathcal{F}}}

\begin{document}


\title[Fr\'echet Derivative]{The Fr\'echet derivative is in the first Baire class}

\author[E. Kopeck\'a]{ Eva Kopeck\'a}
\address{Department of Mathematics\\
   University of Innsbruck\\
 A-6020 Innsbruck, Austria}
\email {eva.kopecka@uibk.ac.at}
\address{and}
\author[L. Zaj\'{\i}\v cek]{Lud\v ek  Zaj\'{\i}\v cek} 
 \address{ Charles University, Faculty of Mathematics and Physics \\
 Sokolovsk\'a 83, 186 75 Praha, Czech Republic}
\email{zajicek@karlin.mff.cuni.cz}


\subjclass[2020]{Primary: 46G05, Secondary: 26B05}
\keywords{Fr\'echet derivative, first Baire class}


\begin{abstract}
Let $X$ be a normed linear space, $Y$ a Banach  space, $G\subset X$ an 
open set and $f:G \to Y$ a  mapping. We show that the Fr\'echet derivative  $f'$ of $f$ is of the first Baire class on the (possibly empty)  set $D\subset G$ where it is defined.
\end{abstract}

\maketitle

\section*{Introduction}

Assume that  $X$ is a normed linear space, $Y$ a Banach  space, $G\subset X$ an 
open set, and $f:G \to Y$ a (not necessarily continuous) mapping. Let  $D$  be  
the set of points $x\in G$ at which the Fr\'echet derivative $f'(x)$ exists. According to \cite{Z} the set  $D$ is always $F_{\sigma\delta}$, and it can, of course,  be empty. In Theorem~\ref{main} we show that the mapping $f':D\to L(X,Y)$ is of the first Baire class. This generalizes a result of \cite{MS}, where 
 it is assumed 
 that $X$ is a Banach space and $f$ is  Fr\'echet differentiable on $G$. Our method of proof differs from that of \cite{MS}.  We construct a sequence of ``step functions" which uniformly converges on $D$ to $f'$. Each of the step functions is of the first Borel class and has  a sigma discrete  base in the sense of 
Hansell, hence it is also in the first Baire class (see Fact~\ref{basef})  and so is $f'$ as its uniform limit. We also present a shorter proof, rather close to the original one, of 
the result of \cite{MS}.    

 Our reason for returning    to this theme is twofold. The paper \cite{MS} appeared only as an ESI-preprint but at the same time its results were used in \cite{HWZ} to give a more transparent proof and a  generalization of
Malý's result on the Darboux property of  Fréchet derivatives. Further, we believe that our generalization
assuming $X$ to be just a  normed linear space and $D$ to be possibly smaller than $G$ is of some interest.
\vspace{2mm}
 
\noindent{\bf Notation and terminology.}
If $\F$ is a family of sets, $\bigcup \F$ denotes the union of all its members.
We denote by $B(c,r)$ the open ball of center $c$ and radius $r$. The symbol $\osc(f,S)$ denotes the oscillation of the mapping $f$ on the set $S$.  The norm on any normed linear space is denoted by $|\cdot|$. 
By $f'$ we {\em always} mean the Fréchet derivative.
  If  $X$ is a  normed linear space  and $Y$  is a Banach space,  we denote by $L(X,Y)$   the Banach space of bounded linear operators from $X$ to $Y$ equipped with the operator  norm.    

A function (= mapping) between metric spaces is of the {\it first Baire class} if it is a pointwise limit of 
 continuous functions. It is of the {\it  first Borel class} if its preimages of open sets are $F_{\sigma}$.
Recall that each function of the first Baire class is also of  Borel class one, but the opposite implication
 holds in special cases only (cf., e.g., \cite{V}).

\section{Basic lemma}

 In this section we formulate and  prove   Lemma~\ref{osc} which is implicitly contained in \cite{MS}.
We  will use the following  elementary observation which appears in  \cite{Z87} in the case where $S=X$. As it significantly shortens the verification of Lemma~\ref{osc} and hence also of Theorem~\ref{MS}, we include its simple proof as well.

\begin{lemma}\label{zah} 
Let $X$, $Y$ be real normed linear spaces  $S\subset X$ and $f:S \to Y$ an arbitrary mapping. Suppose that
 $A:X \to Y$ is a linear mapping, $c\in X$, $\ep>0$ and $\delta>0$ such that
 $|f(c+h)-f(c)- A(h)| \leq \ep |h|$ whenever $|h|< \delta$. Then the inequalities $|x-c|< \delta$,
 $|y-c|< \delta$ and $|x-y|\geq |x-c|$ imply the inequality
$$   |f(y)-f(x)- A(y- x)| \leq 3\ep |y-x|. $$
\end{lemma}
\begin{proof}
By the assumptions we have $|f(x)- f(c)- A(x-c)| \leq \ep|x-c|$ and
$|f(y)-f(c)-A(y-c)|\leq \ep |y - c|$. Consequently,
$$
|f(y)-f(x)-A(y-x)|\leq\ep(|x-c| + |y-c|)\leq \ep(|x-c| + |x-c| +
|y-x|)\leq 3\ep|y-x|.
$$
\end{proof}

In this section, in fact in the  entire paper, we work with the following  
\vspace{3mm}

\noindent{\bf Setting.} We suppose that  $X$ is a normed linear space,  $Y$ is a Banach  spaces, $G\subset X$ an 
 open set and $f:G \to Y$ a (not necessarily continuous) mapping. 
We set
\begin{equation}\label{D}
D:= \{x\in G:\ f'(x)\ \ \text{exists}\}.
\end{equation}
Further we set, for  $m,n \in \N$,
\begin{equation}\label{ank}
A_{n,m}:= \{x\in D:\ |f(y)-f(x)-f'(x)(y-x)| \leq \frac{1}{n}|y-x| \text{ whenever } |y-x| < \frac{1}{m}\}.
\end{equation}
\vspace{3mm}

The set $D$ is always an $F_{\sigma\delta}$ set according to   \cite[Theorem 2]{Z}  (see also
\cite[Corollary 3.5.5]{LPT}
 for a different proof).

Obviously
\begin{equation}\label{danm}
D= \bigcup_{m\in \N} A_{n,m}\ \ \ \text{ for each } n\in \N
\end{equation}
and
\begin{equation}\label{anm}
A_{n,m} \subset A_{n,m^*}\ \ \text{ whenever } m \leq m^*.
\end{equation}

\begin{lemma}\label{osc}
Assume Setting.
\begin{enumerate}
\item[(i)]
If  $x_1, x_2 \in A_{n,k}$ and  $|x_1-x_2|< 1/(2k)$, 
then $|f'(x_1)-f'(x_2)|\leq 4/n$.
\item[(ii)]
If $U\subset X$ is open with $\diam U< 1/(2m)$, then  
$\osc(f', D \cap U\cap \overline{A_{n,m}})  \leq 12/n$.
\end{enumerate}
\end{lemma}
\begin{proof}
{\bf (i):}\ 
Choose an arbitrary $v\in X$ with $|v|=1$ and set $y:= x_2 + 1/(2k) \cdot v$. Since  $x_2 \in A_{n,k}$,
 we have
\begin{equation}\label{x2}
|f(y)-f(x_2) - f'(x_2) ( 1/(2k) \cdot v )| \leq \frac{1}{2kn}. 
\end{equation}
  We have $|x_1-x_2|< 1/ (2k)< 1/k$,  $|y-x_2|=  1/ (2k)$, $|x_1-y| \leq |x_1-x_2| + |y-x_2|
< 1/ (2k) + 1/ (2k) = 1/k$ and $|y-x_2| > |x_1-x_2|$. Since $x_1 \in A_{n,k}$,  we  can apply
  Lemma~\ref{zah}  with $c= x_1$, $x= x_2$, $\ep=1/n$, 
$\delta= 1/k$, and $S=G$ and obtain
$$
|f(y)-f(x_2) - f'(x_1) (1/(2k) \cdot  v)| \leq  \frac{3}{2kn}.
$$

Using also \eqref{x2}, we
	 obtain 
$$
|f'(x_1) (1/(2k) \cdot v) - f'(x_2) (1/(2k)\cdot v)| \leq \frac{4}{2kn}.
$$ 
Therefore
	$|(f'(x_1)- f'(x_2)) (v)|\leq  4/n$  and $|f'(x_1)- f'(x_2)|\leq 4/n$ follows. \\
{\bf (ii):} We can assume that $D \cap U \cap \overline{A_{n,m}} \neq \emptyset$. Let 
	 $z_1, z_2 \in D \cap U \cap \overline{A_{n,m}}$ be arbitrary.    By \eqref{danm} and \eqref{anm} we can choose  $k>m$ such that $z_1\in A_{n,k}$. Obviously,
	 we can choose $x_1 \in A_{n,m} \cap U$ with $|x_1-z_1|< 1/(2k)$. Since $x_1 \in  A_{n,k}$, we obtain
	 by (i) that $|f'(x_1)-f'(z_1)| \leq 4/n$. In the same way we  find $x_2 \in A_{n,m} \cap U$
	 with $|f'(x_2)-f'(z_2)| \leq 4/n$. Since  $\diam U < 1/(2m)$, we have $|x_1-x_2|< 1/(2m) $. 
Then  (i) used with $k:=m$ gives $|f'(x_1)- f'(x_2)|\leq 4/n$. Now $|f'(z_1)-f'(z_2)|\leq  12/n$
	 easily follows.
\end{proof}

\section{The case where $X$ is a Banach space and $D=G$}

In this section we use Lemma~\ref{osc} to give a very short proof of the following  result
 of \cite{MS}.
\begin{theorem}\label{MS} \cite[Theorem 1]{MS}
Let $X$ and $Y$ be Banach spaces, $G \subset X$ an open set and $f: G \to Y$ a Fréchet differentiable function.
 Then $f':G \to L(X,Y)$ is a function of the first Baire class.
\end{theorem}
Our proof is very close to that of \cite{MS}. Instead of 
\cite[Theorem 4, $(ii) \Rightarrow (i)$]{S} we use the following result which is an immediate consequence
 of \cite[Theorem 3]{S}.
\begin{proposition}\label{S}\cite{S}
Suppose that $X$ is a complete metric space, $Y$ is a Banach space
and $\varphi : X \to  Y$ is a function with the following property:

for each $\epsilon > 0$ and each nonempty  closed subset $C$ of $X$   there exists an open subset $U \subset  X$ with $U\cap C\neq \emptyset$  such that the oscilation $\osc(\varphi, U\cap C)<\ep$.

Then $\varphi$ is a Baire one function.
 \end{proposition}
To prove the above proposition, it is sufficient to apply   \cite[Theorem 3]{S} to the ``multivalued'' mapping
 $\Phi(x):= \{\varphi(x)\},\ x \in X$.
 \vspace{3mm}

{\bf Proof of Theorem~\ref{MS}.}
We equip   $G$ with an equivalent complete metric $\sigma$. We  apply Proposition~\ref{S}  to
 $\varphi:= f':(G,\sigma) \to L(X,Y)$. Choose an  arbitrary $\emptyset \neq C \subset G$ closed in
 $(G,\sigma)$ and $\ep>0$. Choose $n \in \N$ with $12/n < \ep$. Since $D=G$, by \eqref{danm} we have  
$$  C= \bigcup_{m\in \N} (C \cap A_{n,m}).$$
Since $(C,\sigma)$ is complete, by the Baire theorem we can choose $m\in \N$ such that
$C \cap \overline{A_{n,m}}$ has nonempty interior in $(C,\sigma)$. Then we can choose
 $w \in C$ and $\delta< 1/(4m)$ such that, denoting $U:= B(w,\delta)$, the ball is in $X$, we have
 $U\subset G$ and $U\cap C \subset  \overline{A_{m,n}}$. By Lemma~\ref{osc} (ii) we have
 $\osc(f', C \cap U)  \leq 12/n < \ep$. We conclude by    Proposition~\ref{S} that $f'$ 
 is a function of the first Baire class on $(G,\sigma)$ and so also on $G$ with the norm metric.

\section{The general case}

In this section we  prove the following generalization of Theorem~\ref{MS}.

\begin{theorem}\label{main}
Suppose that  $X$  is a normed linear space, $Y$ is a Banach  space, $G\subset X$ is an 
open set and $f:G \to Y$ is an arbitrary function. Let  $D$  be  
the set of points $x\in G$ at which the Fr\'echet derivative $f'(x)$ exists.  Then    $f':D\to L(X,Y)$ is of the  first Baire class.
\end{theorem}

For the proof we substantially use the notion of a $\sigma$-discrete family  and its properties.

Recall that a 
  family of sets in a topological space is said to be {\em discrete} if each
point of the space has a neighbourhood that meets at most one set of the family.
A family of sets is said to be {\em $\sigma$-discrete} if it is a union of countably
many subfamilies, each of which is discrete.

We  use the following well-known facts.

\begin{fact}\label{discr}
Let $X$ be a metric  space. 
\begin{enumerate}
\item [(i)]
If $\F$ is a  discrete (resp. $\sigma$-discrete) family of subsets of $X$, then each subfamily of $\F$
 is  discrete (resp. $\sigma$-discrete) as well.
\item [(ii)]
If $\F$ is a  discrete  family of subsets of $X$ and $M\subset X$, then the family 
 $\{F \cap M:\ F \in \F\}$ is discrete as well.
\item [(iii)]
Let $\F$ be a discrete family of $F_{\sigma}$ (resp. $ G_{\delta}$) subsets of $X$. Then
 the union of all its sets $\bigcup \F$ is $F_{\sigma}$ (resp. $ G_{\delta}$) as well.
\item [(iv)]
Let $\F$ be a $\sigma$-discrete family of $F_{\sigma}$  subsets of $X$. Then
   $\bigcup \F$ is $F_{\sigma}$ as well.
\item[(v)]
For each $\ep>0$ there exists a $\sigma$-discrete cover $\B$ of $X$ by open sets of diameters less than $\ep$.	
\end{enumerate}
\end{fact}
\begin{proof}
The proofs of  (i) and (ii) are  straightforward. For (iii) see e.g.  \cite[p. 148]{Ha} and (iv) is an easy consequence of (iii).
To prove (v), consider the  cover of $X$ by open balls of radius $\ep/2$ and define $\B$ as
 its  $\sigma$-discrete refinement, 
 which exists by e.g. \cite[p. 234]{K}.
\end{proof}

By \cite{Ha},  a family of sets is called a {\em base for a function} from one topological space into
another if the preimage of any open set is a union of sets from the family.
A function is said to be $\sigma$-discrete if it has a $\sigma$-discrete base. 

We  need the following fact.

\begin{fact}\label{basef}
Let $D$ be a metric space, $Z$ a Banach space  and let $\varphi:D\to Z$ be a  function.  If $\varphi$ is of the first Borel class and, moreover, has a $\sigma$-discrete base, then $\vf$ is also of the first Baire class.
\end{fact}
This fact follows immediately from \cite[Theorem 2]{F} but also from \cite[Theorem 1, Theorem 2 (b)]{Rog}. Further relevant information
 can be found in \cite{V} and \cite{Rog}.

Another important ingredient of our proof is the following well-known fact.

\begin{fact}\label{unif}
Let $D$ be a metric space, $Z$ a Banach spaces and $f_n:D \to Z$ functions of the first Baire class.
 If the sequence $\{f_n\}$ converges uniformly to a function $f:D\to Z$, then $f$ is of the first Baire class
 as well.
\end{fact}
An essentially classical proof of this fact  is sketched in \cite{S}, p. 984. It follows  also from  a more general \cite[Lemma 1]{Rol}.

To prove Theorem~\ref{main} it is sufficient, by Fact~\ref{unif}, to
 construct a sequence   $\{\vf_n\}$ which uniformly converges on $D$ to $f'$ and each $\vf_n$
  is of the first Baire class.  Our $\vf_n$'s are  some ``step functions"; $\vf_n$ is constant on
	 each set from a  $\sigma$-discrete  family $\Omega_n$ of $F_{\sigma}$ sets, which  is a partition 
	 of $D$. This implies that $\vf_n$ is  of the first Borel class. Observing that $\Omega_n$
	 is a base of $\vf_n$ and using Fact~\ref{basef}, we obtain that   $\vf_n$ is  even of the first Baire class.

The proof of Theorem~\ref{main}   is quite easy using the following lemma. For the proof of the lemma we will need
 Lemma~\ref{osc} and Fact ~\ref{discr}.

\begin{lemma}\label{partition}	
Suppose that  $X$ is a normed linear space,  $Y$ is  a Banach  space, $G\subset X$ is an 
open set and $f:G \to Y$ is an arbitrary  mapping. Let  $D$ be the set  of the points $x\in G$ at which the Fr\'echet derivative of $f$ exists.
	Then for each $\varepsilon>0$ there exists a family $\Omega$ of subsets of $D$ with the following
	 properties:
\begin{enumerate}
\item[(a)]	The sets in $\Omega$  are pairwise disjoint.
\item[(b)] The union of all of the sets in $\Omega$  is $D$.
\item[(c)]	The family $\Omega$   is $\sigma$-discrete.
\item[(d)]	Each member of	$\Omega$  is an $F_{\sigma}$ set  in $D$. 
\item[(e)]	For each $ E \in \Omega$  it holds that $\osc(f',E) \leq \varepsilon$.
\end{enumerate}
\end{lemma}
\begin{proof}
	 By  Fact~\ref{discr} (v) we  can choose for each $m\in \N$ a $\sigma$-discrete cover
 $\B^m$ of $X$ by open sets of diameters
less than $1/(2m)$ 
and write $\B^m = \bigcup_{i=1}^{\infty} \B^m_i$
where each $\B^m_i$ is a discrete family.	
Choose $n\in \N$ such that $1/(12n) < \varepsilon $ and define sets $A_{n,m}$ by \eqref{ank}.
Then set  $A^*_m:= \overline{A_{n,m}} \cap D$ for $m\in \N$ and $A_0^*:= \emptyset$, and define
$$  \Omega_i^{m}:= \{ B \cap  A_m^*:\ B \in \B^m_i\},\ \ m,i \in \N. $$ 
 Note that the family $\Omega^*:= \bigcup_{m,i \in \N} \Omega_i^{m}$ has properties (b)-(e) but need not
 be disjoint, so we will modify $  \Omega_i^{m}$.
Namely, we first set,  for $i,m \in \N$,
\begin{equation}\notag
  S_i^m:= \bigcup \Omega_i^{m},\ \ \ \ \  \hat S_i^m:= \bigcup_{1\leq j\leq i} S_j^m, \ \ \  \ \ 
   \hat S_0^m:= \emptyset, 
\end{equation}
and then
\begin{equation}\notag
\begin{split} 
\tilde \Omega_i^m&:= \{ Z \setminus (\hat S_{i-1}^m \cup A^*_{m-1}):\ Z \in \Omega_i^m\},\ \ i,m\in \N, \\ 
 \Omega&:= \bigcup_{m,i \in \N} \tilde \Omega_i^m.
\end{split}
\end{equation}
First we  show that the  family $\Omega$ has properties (a) and  (b) meaning  that it is a partition of $D$.
Using \eqref{danm} and \eqref{anm}  we easily see that the family
\begin{equation}\label{decahv}
\{ A_m^* \setminus A^*_{m-1}:\ m \in \N\}\ \ \ \ \text{is a partition of } D. 
\end{equation}

Since $\B^m$ is a cover of $X$, we obtain  $\bigcup_{i\in \N} S_i^m =  \bigcup_{i\in \N} \hat S_i^m =  A^*_m,
\ m\in \N$,
 and so, for each $m\in \N$, the family
 \begin{equation}\label{dechs}
\{ \hat S_i^m \setminus  \hat S_{i-1}^m:\ i \in \N\}\ \ \ \ \text{is a partition of } A^*_m. 
\end{equation}
Since $\B_i^m$ is disjoint (even discrete), $\Omega_i^m$ is disjoint as well; so it is a partition
 of   $S_i^m$. Consequently $\tilde \Omega_i^m$ is a partition of the set
 $(S_i^m \setminus \hat S_{i-1}^m) \setminus A^*_{m-1} =
 (\hat S_i^m \setminus \hat S_{i-1}^m) \setminus A^*_{m-1}$.
Therefore, using \eqref{dechs}, we get that the family
$\bigcup_{i\in \N} \tilde \Omega_i^m$ is a partition of $ A^*_m \setminus A^*_{m-1}$.
Hence  $\Omega= \bigcup_{m\in \N} \bigcup_{i\in \N} \tilde \Omega_i^m$ is a partition
 of $D$ by \eqref{decahv}.

Since each family $\B_i^m$ is discrete,  by Fact~\ref{discr} (ii) all  $\Omega_i^m$ and $\tilde \Omega_i^m$ 
are 
 discrete as well, and  (c) follows.

Further,  each $\Omega_i^m$ is discrete and each one of  its members is clearly both $F_{\sigma}$
 and $G_{\delta}$ in $D$. Therefore by   Fact~\ref{discr} (iii) all $S_i^m$, and so also  $\hat S_i^m$,
 are $G_{\delta}$ in $D$. Thus we easily infer that each member of each $\tilde \Omega_i^m$ is
$F_{\sigma}$ in $D$ and (d) follows.

 To prove (e), consider an arbitrary $E \in \Omega$. Then $E\in \tilde \Omega_i^m$ for some $m,i \in \N$.
 By the  definitions of $\tilde \Omega_i^m$ and $\Omega_i^m$ we obtain that 
$E \subset A^*_m = \overline{A_{n,m}} \cap D$. At the same time $E\subset U$ for some $U \in \B^m$ with $\diam U < 1/(2m)$.
Hence $\osc(f',E)  \leq  1/(12n) < \ep$ by
Lemma~\ref{osc}. 
\end{proof}

{\bf Proof of Theorem~\ref{main}}

Assume $D$ is not empty. For $n\in \N$ let $\Omega_n$ be as $\Omega$ in Lemma~\ref{partition} for $\varepsilon=1/n$. We define a function $\vf_n:D\to L(X,Y)$ as follows. For each nonempty set $E\in \Omega_n$ fix a point $x_E\in E$ and define $\vf_n$ on $E$ as the constant function equal to $f'(x_E)$. According to (a) and (b) of   Lemma~\ref{partition} the function $\vf_n$ is well defined. 
 Further, Lemma~\ref{partition} (e) implies that
\begin{equation}\label{blizko} 
 | \vf_n(x)-f'(x)|
 \leq 1/n\ \ \ \text{whenever}\   x \in D.
\end{equation}
For each open set $H \subset L(X,Y)$ we have
\begin{equation}\label{pre}
(\vf_n)^{-1} (H) = \bigcup \{E \in \Omega_n:\ f'(x_E) \in H\}. 
\end{equation}
  By   Lemma~\ref{partition} (c) and Fact~\ref{discr} (i) the family $\{E \in \Omega_n:\ f'(x_E) \in H\}$ is $\sigma$-discrete.
Using this fact, \eqref{pre}, Lemma~\ref{partition}  (d) and Fact~\ref{discr}  (iv), we obtain that
 $(\vf_n)^{-1} (H)$ is $F_{\sigma}$; hence
 the function $\vf_n$ is of the first Borel class.   By \eqref{blizko} the sequence $\{\vf_n\}$   
converges uniformly    on $D$ to $f'$ and hence $f'$ is also of the  first Borel class (\cite{K}, p. 386). We can do even better. By  \eqref{pre}  the $\sigma$-discrete family $\Omega_n$ is a base of
 $\vf_n$  and so by Fact~\ref{basef} each  $\vf_n$ is also of the first Baire class. Consequently 
   $f':D\to L(X,Y)$ is  of  the first Baire class  by Fact~\ref{unif}.

\end{document}